\documentclass[12pt, a4paper]{article}
\usepackage{mathrsfs, amssymb, eucal, amsfonts, amsmath, amsthm, amscd}
\numberwithin{equation}{section}

\newcommand{\Bc}{{\mathcal B}}
\newcommand{\Dc}{{\mathcal D}}\newcommand{\Ec}{{\mathcal E}}\newcommand{\Fc}{{\mathcal F}}

\newcommand{\Sc}{{\mathcal S}}

\newcommand{\As}{{\mathscr A}}
\newcommand{\Cs}{{\mathscr C}}

\newcommand{\bFc}{\boldsymbol{\Fc}}\newcommand{\bF}{{\mathbf F}}
\newcommand{\Hs}{{\mathscr H}}

\newcommand{\Mg}{{\mathfrak M}}

\newcommand{\R}{{\mathbb R}}

\newcommand{\bnu}{\boldsymbol{\nu}}

\begin{document}
\newtheorem {defn}{Definition}[section]
\newtheorem{teo}{Theorem}[section]
\newtheorem{lem}[teo]{Lemma}
\newtheorem{prop}[teo]{Proposition}
\newtheorem{cor}[teo]{Corollary}
\theoremstyle{remark}
\newtheorem{ex}{Example}[section]
\newtheorem{rem}{Remark}[section]
\hyphenation{se-cond}
\hyphenation{o-pe-ra-tor}
\hyphenation{de-ve-lop-ments}
\hyphenation{sta-tio-na-ri-ty}
\hyphenation{self-ad-jointness}
\hyphenation{po-si-ti-ve}

\title{Stochastic Mappings and Random Distribution Fields II. Stationarity}

\author{P\u{a}storel Ga\c{s}par \and Lorena Popa}

\date{\today}

 \maketitle


\begin{abstract}
As a continuation of \cite{GaPopa} this paper treats the stationary and stationarily cross-correlated multivariate stochastic mappings. Moreover for the case of multivariate random distribution fields, a particular form for the operator cross covariance distribution is given, from which a Kolmogorov type isomorphism theorem and a spectral representation of a stationary multivariate random distribution field are derived.
\end{abstract}

\section{Introduction}
The stationarity for the univariate, one time parameter stochastic processes was introduced by Khincine \cite{Khin} and developed by Kolmogorov \cite{Kolmo}. Other important results on finite variate case were obtained by Wiener and Masani (\cite{WieMas1}, \cite{WieMas2} and \cite{MasPred}). \\
The study of stationarity for  infinite variate  random fields is  treated in Section 4.2 of the book \cite{Kaki}. \\
The first extension of stationarity from one continuous time parameter univariate stochastic processes regarded as continuous functions to such stochastic processes regarded as distributions was given first by It\^o \cite{Ito} and Gelfand \cite{Gelf1} and continued in \cite{Roz2} and
\cite{Bala}. \\
On the other side, following the framework in \cite{Kaki}, an extension of a general correlation theory to the distributional setting, but for the was given in our recent paper \cite{GaPopa}. By using the development from \cite{GaPopa} it is our aim to treat the stationarity in this setting here. \\
Since the definition of the stationarity from \cite{Kaki} is not applicable in the distributional setting chosen, as indirectly suggested in \cite{Ito}, it is necessary to make use of the action of the group of translations of the euclidian space $\R^d$ ($d$ being a fixed natural number) on the fundamental space $\Dc_d$ from distribution theory.
For an adequate placement in our development of \emph{multivariate stochastic measures} we shall define the stationarity in the most general setting from \cite{GaPopa}, namely that of \emph{multivariate stochastic mappings}.\\
Consequently, we shall currently use the notations and often refer to results from \cite{GaPopa}.\\
Hence, for a separable infinite dimensional Hilbert space and a probability space $(\Omega, \As, \wp)$, $L^2_{s,0}[\wp, H] =\ : \Hs$ means the normal Hilbert $\Bc(H)$-module of $H$~-~valued strong second order random variables of zero mean with the gramian denoted by $[\cdot,\cdot]_{\Hs}$, while $\Dc_d$ is the space of complex valued test functions on the $d$~-~dimensional real euclidian space $\R^d$ from distribution theory and $\Dc'_d(\Hs)$ is the $\Bc(H)$-module of $\Hs$-valued distributions on $\R^d$, whose elements are called \emph{multivariate second order random distribution fields (denoted briefly m.s.o.r.d.f.)}. More generally, when instead of $\Dc_d$ one considers an arbitrary set (a topological space) $\Lambda$, a (continuous) map $\Phi$ from $\Lambda$ to $\Hs$ is called a \emph{multivariate second order stochastic mapping (m.s.o.s.m)} on $\Lambda$, their set being denoted by $\mathbf{M}(\Lambda, \Hs)$ as in Section 2 of \cite{GaPopa}.\\
We also recall from \cite{GaPopa} the notions of operator cross covariance function
$$\Gamma_{\Phi,\Psi}(\lambda,\mu):=[\Phi(\lambda),\Psi(\mu)]_{\Hs} ;\,\quad\lambda,\mu\in\Lambda,$$
(operator cross covariance distribution
$$C_{U,V}(\varphi\otimes\psi):=\Gamma_{U,V}(\varphi,\overline{\psi}), \,\quad\varphi, \psi\in\mathcal{D}_{d}),$$
for a pair $\Phi, \Psi\in \mathbf{M}(\Lambda, \Hs)$ (respectively $U, V\in\mathcal{D}'_{d}(\Hs)$), in whose terms modular time domains and the measurements spaces, as well as the operator subordination (Theorems 2.1, 2.2 and Corollaries 3.2~-~3.4)are described.\\
In section 2, when a semigroup $S$ acts on a set $\Lambda$, or if $S=\Lambda$ and is a $\ast$~-~semigroup, three kinds of $S$-stationarity and $S$-stationarily cross correlatedness are defined, some significant particular cases being discussed. \\
Section 3 is devoted to the (operator) $\R^d$~-stationarity and $\R^d$-stationarily cross correlatedness of m.s.o.r.d.f., where a particular form of the operator cross covariance distribution, the spectral representation of this operator covariance distribution as well as the spectral representation of m.s.o.r.d.f. itself are given .

\section{Preliminaries}

As in \cite{MasDilat} we say that a multiplicative semigroup $S$ \textit{acts} on a set $\Lambda$, if there exists a binary operation
\begin{equation}\label{1420}
S\times \Lambda\ni (s,\lambda)\mapsto s\odot \lambda \in \Lambda,
\end{equation}
with the property
\begin{equation}\label{1421}
(s_{1} s_{2})\odot \lambda = s_{1}\odot(s_{2}\odot \lambda)\,,\quad  s_{1},s_{2}\in S, \lambda\in \Lambda.
\end{equation}
If $S$ has a unit $e$ then it is also required that
\begin{equation}\label{14211}
e\odot \lambda = \lambda\,,\quad \lambda\in \Lambda.
\end{equation}
Given a semigroup $S$  which acts on $\Lambda$, then two m.s.o.s.m. $\Phi$ and $\Psi$ with the same index set $\Lambda$ are called \emph{operator} $S$~-~\emph{stationarily cross correlated}, if their operator cross covariance function $\Gamma_{\Phi,\Psi}$ satisfies
\begin{equation}\label{1424}
\Gamma_{\Phi ,\Psi}^{s}(\lambda_{1},\lambda_{2}):=\Gamma_{\Phi,\Psi}(s\odot \lambda_{1}, s\odot \lambda_{2})=\Gamma_{\Phi,\Psi}(\lambda_{1},\lambda_{2}),\  s\in S,\ \lambda_{1},\lambda_{2}\in \Lambda.
\end{equation}
When (\ref{1424}) holds with the scalar cross covariance function $\gamma_{\Phi , \Psi}$ instead of $\Gamma_{\Phi, \Psi}$, then $\Phi$ and $\Psi$ are said to be $S$-\emph{stationarily cross correlated}. The pair $\Phi, \Psi$ is said to be \emph{scalarly} $S$-\emph{stationarily cross correlated} if, for each $x\in H$,
\textit{ the univariate second order stochastic mappings $\{(\Phi(\lambda),x)_{H}\}_{\lambda\in\Lambda}$, and $\{(\Psi(\lambda),x)_{H}\}_{\lambda\in\Lambda}$ are $S$-stationarily cross correlated.}\footnote[1]{Since for a fixed $\lambda\in\Lambda$, we can regard $\Phi(\lambda)\in L^2_{s,o}[\wp, H]$ also as a $H$-valued function of $\omega\in\Omega$, it is clear that for each $x\in H$, $(\Phi(\lambda),x)_{H}(\omega)$ is defined $\wp$~-~a.e. by $(\Phi(\lambda)(\omega),x)_{H},\, \omega\in\Omega$.}
It is clear that the operator $S$-stationarily cross correlatedness implies the other two kinds of $S$-stationarily cross correlatedness for a pair $\Phi$, $\Psi$. \\
When the operator covariance distribution $\Gamma_{\Phi}:=\Gamma_{\Phi,\Phi}$ satisfies (\ref{1424}), we say that the m.s.o.s.m. $\Phi$ is \textit{operator} $S$-\textit{stationary}. The $S$-stationary and scalarly $S$-stationary m.s.o.s.m. will be in an obvious way correspondingly defined.

In the particular case when $S=\Lambda$, the $S$-stationarily cross corelatedness will be defined only when $S$ is a $\ast$-semigroup (i.e. endowed with an involution $\ast$).
Namely a pair $\Phi,\Psi\in \mathbf{M}(S,\Hs)$ is called (operator) $S$-stationarily cross correlated when the (operator) cross correlation function $\gamma_{\Phi,\Psi}(s,t)$, (respectively $\Gamma_{\Phi,\Psi}(s,t)$) do not depend separately on $s$ and $t$, but of $st^{\ast}$ and when $S$ is a topological $\ast$~-~semigroup the function $\widetilde{\Gamma}_{\Phi,\Psi}:S\rightarrow\Hs$, satisfying $\widetilde{\Gamma}_{\Phi,\Psi}(st^{*})=\Gamma_{\Phi,\Psi}(s,t)$ is continuous on $S$. In an obvious way will be defined the scalarly $S$-stationarily cross  correlatedness as well as the corresponding three kinds of $S$-stationarity for a m.s.o.s.m. $\Phi$.\\
Let us remark that when $S$ is a locally compact abelian group with the inversion of elements as involution, then the above definitions are exactly that one from \cite{Kaki}(Definition 1 Sec 4.2 pp 151), which are also obtained as particular case of $G$-stationarities considering the action of $G$ upon itself trough its translations.\\
It is well known that in the spectral representation of operator stationary m.s.o.r.f. an important role is played by the multivariate second order stochastic regular (bounded) measures (see Theorem 2 Sect 4.2 pp. 151 from \cite{Kaki}). As we shall see in the next Section the class of multivariate second order stochastic regular (not necessarily bounded) measures will play a similar role in the spectral representation of operator $\mathbb{R}^{d}$-stationary m.s.o.r.f.d.\\
This is why in the rest of this Section we shall consider the (operator) stationarity and the (operator) stationarily cross correlatedness for m.s.o. stochastic regular (not necessarily bounded) measures on $\mathbb{R}^{d}$ from the classes $f o s v r \mathcal{M}_{d}(\Hs)$ and $f s v r \mathcal{M}_{d}(\Hs)$ considered in Section 4 (Prop.4.1) from \cite{GaPopa}. As we have seen the index set for these particular m.s.o.s.m. is the $\delta$-ring $\mathcal{B}or(\mathbb{R}^{d})$ of bounded Borel sets from $\mathbb{R}^{d}$, which will be regarded as a $\ast$-semigroup with the intersection as inner operation and with the identity as involution (i.e. $A^{\ast}=A, A\in \mathcal{B}or(\mathbb{R}^{d})$)\\
Let us remark that to a pair $\xi, \eta \in f(o) s v r \mathcal{M}_{d}(\Hs)$ the (operator) cross covariance function defines a $\Cs_{1}(H)$-valued bimeasure $\tau_{\xi,\eta}\in f(o)svr \Mg_{2d}(\Cs_{1}(H))$ trough
$$\tau_{\xi,\eta}(A,B)=(\xi\otimes\eta)(A,B)=[\xi(A),\eta(B)]_{\Hs}\,,\quad A,B\in\mathcal{B}or(\mathbb{R}^{d}),$$
while the (operator) covariance bimeasure of $\xi$ is $\tau_{\xi}=\tau_{\xi,\xi}$, wich is a positive definite bimeasure (see also Section 4 from \cite{GaPopa}), i.e. with the notation from \cite{GaPopa}, $\tau_{\xi}\in f(o)svr \Mg_{2d}^{pd}$. Now we say that the pair $\xi, \eta$ is operator $\mathcal{B}or(\mathbb{R}^{d})$-stationarily cross correlated when there is a $\Cs_{1}(H)$-valued (not necessarily bounded) measure $\zeta\in \mathcal{B}or(\mathbb{R}^{d})$ such that $\tau_{\xi,\eta}=\zeta(A\cap B) \, , \quad A, B \in \mathcal{B}or(\mathbb{R}^{d})$.

\section{Stationary multivariate random distribution fields}

In terms of the general concepts introduced in the previous Section, we study here different kinds of stationarity for m.s.o.r.d.f. in the framework developed in \cite{GaPopa}, extending simultaneously the stationarity of one-time parameter random distributions from \cite{Ito}, \cite{Gelf1} and the different types of stationarity of m.s.o.r.f. on $\R^d$  as treated  for infinite variate second order stochastic processes on an arbitrary locally compact abelian group $G$ in Sec. 4.2. of \cite{Kaki}.

To this purpose we shall use two types of stationarity, for which we shall prove that they are equivalent. The \textit{first one} is that obtained by the outer action upon $\Dc_{d}$ of the group $\R^d$ given by
\begin{equation}\label{441}
\mathbb{R}^{d}\times \Dc_{d}\ni (x,\varphi)\mapsto x\odot \varphi\:=\tau_{x}\varphi \in\Dc_{d},
\end{equation}
where $\tau_{x}\,,(x\in\mathbb{R}^{d})$ is the translation operator, i.e. $(\tau_x \varphi)(y)= \varphi(y-x),\ y \in \mathbb{R}^d$.
The second one will by obtained by considering $\Dc_d$ as a $\ast$~-~semigroup (without unit) with
the convolution as inner multiplication and the involution ${ }^\sim$ (i.e. $\widetilde \varphi (x) = \overline{\varphi(-x)}, x \in \R^d,\ \varphi \in \Dc_d$).
\begin{rem}
Let us observe that the convolution on $\Dc_d$
\begin{equation}\label{eq:act_D_d_on_D_d}
\Dc_d\times\Dc_d\ni (\psi, \varphi) \mapsto \varphi*\psi \in \Dc_d
\end{equation}
and the action \eqref{441}
are linked by the following invariance of the convolution with respect to the action \eqref{441}
\begin{equation}\label{eq:link_R^d_D_d}
(x\odot \varphi)*(\widetilde {x\odot \psi}) = \varphi * \widetilde\psi,
\end{equation}
for each $x\in \R^d$ and $\varphi, \psi \in \Dc_d$.
\end{rem}

Putting $U, V, \mathbb{R}^d, \mathcal{D}_d$, respectively $U, V, \Dc_d, \Dc_d$ instead of $\Phi, \Psi,  S, \Lambda$ in the definition of the (operator) cross $S$-stationarity
for $\Phi$ and $\Psi$, we have the definition of the \textit{(operator) cross} $\R^d$-\textit{stationarity}, respectively the \textit{(operator) cross } $\Dc_d$-\textit{stationarity for the m.s.o.r.d.f. $U$ and $V$}. Of course if $U = V$ we have the definition of the (\textit{operator}) $\R^d$-\textit{stationarity}, respectively the (\textit{operator}) $\Dc_d$-\textit{stationarity of} $U$.\\

\begin{rem}
The definition of the (operator) $\mathbb{R}^{d}$-stationarity for a m.s.o.r.d.f. $U=\{U_{\varphi}\}_{\varphi\in \Dc_{d}}$ is coherent to the classical definition of the (operator) stationarity of a multivariate random field.

Indeed, if $F:\mathbb{R}^{d}\rightarrow \Hs$ is a multivariate second order random field in the terminology adopted in \cite{GaPopa} and in the framework of \cite{Kaki}, and $U_{F}$ is the associated multivariate random field distribution (given in (4.1) of \cite{GaPopa}), then the operator $\mathbb{R}^{d}$-stationarity of $U_{F}$ means that for any $x\in \mathbb{R}^{d}$ and any $\varphi,\,\psi\in\Dc_{d}$

$$\Gamma_{U_{F}}(x\odot\varphi,x\odot\psi)=\Gamma_{U_{F}}(\varphi,\psi),$$
which is equivalent to
\begin{multline*}
\left[\int\varphi(y-x)F(y)dy,\int\psi(u-x)F(u)du\right]_{\Hs}= \\
\left[\int\varphi(y)F(y)dy,\int\psi(u)F(u)du\right]_{\Hs},
\end{multline*}
that is
\begin{multline*}
\int\int\varphi(y-x)\overline{\psi(u-x)}\big[F(y),F(u)\big]dydu = \\
\int\int\varphi(y)\overline{\psi(u)}\big[F(y),F(u)\big]dydu.
\end{multline*}
\end{rem}

Now after a rather obvious change of variable in the left hand side, then considering instead of $\varphi$ and  $\psi$, two regularizing sequences that converge to the Dirac distribution concentrated in $y_{0}$ and $u_{0}$, respectively and finally passing to the limit, one obtains

$$\left[F(y_{0}+x),F(u_{0}+x)\right]_{\Hs}=\left[F(y_{0}),F(u_{0})\right]_{\Hs}$$
where $x, y_{0},u_{0}$ are arbitrary in $\mathbb{R}^{d}$. For $x=-u_{0}$ we have  exactly the operator stationarity of $F$ known from \cite{Kaki} (Def 1 Section 4.2, pp. 151). Conversely, since for each $x\in\mathbb{R}^{d}$ we have $\tau_{x}U_{F}=U_{\tau_{x}}F$ it results easily that the classical operator stationarity of the m.s.o.r.f.  $F$ implies the $\mathbb{R}^{d}$~-~stationarity of the m.s.o.r.d.f. $U_{F}$.\\
The coherence of the stationarity and scalar stationarity of $U_{F}$ to that of $F$ is handled analogously, the same being true for (operator) cross stationarity.

\begin{rem}
The notions of stationarity introduced above can be formulated in terms of the scalar (or operator) cross covariance distributions $c_{U,V}\in \Dc'_{2d}$, respectively $C_{U,V}\in \Dc'_{2d}(\Cs_{1}(H))$  associated in \cite{GaPopa} to the pair $U, V$ of m.s.o.r.d.f.. Thus extending the action \ref{441}, to an action of $\mathbb{R}^d$ upon $\Dc_{2d}$

\begin{equation}\label{eq:action.squaredot}
\mathbb{R}^{d}\times \Dc_{2d}\ni (x,\chi)\mapsto \tau_{(x,x)}\chi=x\boxdot \chi \in\Dc_{2d},
\end{equation}
where if $x=(x_{1},\ldots,x_{d})\in\mathbb{R}^{d}$, then  $(x,x)=(x_{1},\ldots,x_{d},x_{1},\ldots, x_{d})\in\mathbb{R}^{2d}$, the conditions for (operator) cross $\mathbb{R}^{d}$-stationarity of $U, V$ become
\begin{eqnarray}
\tau_{(x,x)}c_{U, V}&=c_{U,V}\,;\quad x\in\mathbb{R}^{d}\,,\\
\tau_{(x,x)}C_{U,V}&=C_{U,V}\,;\quad x\in\mathbb{R}^{d}\,.
\end{eqnarray}
\end{rem}

In what follows we shall see that, under these kinds of cross stationarity conditions, the (operator) cross covariance distributions can be also expressed as (positive definite, if $U=V$) distributions on $\mathbb{R}^d$. 

\begin{teo}\label{teo441}
If $U,V\in \Dc'_{d}(\Hs)$ are (operator) $\mathbb{R}^{d}$-stationarily cross correlated m.s.o.r.d.f., then there exists a (operator) scalar distribution in $d$ variables $k_{U,V}$ ($K_{U,V}$ respectively), for which
\begin{equation}\label{eq:corel.scal}
\gamma_{U,V}(\varphi,\psi)=c_{U,V}(\varphi\otimes\overline{\psi})=
k_{U,V}(\varphi\ast\widetilde{\psi})\,;\quad\varphi,\psi\in\Dc_{d},
\end{equation}
respectively
\begin{equation}\label{eq:corel.oper}
\Gamma_{U,V}(\varphi,\psi)=C_{U,V}(\varphi\otimes\overline{\psi})=
K_{U,V}(\varphi\ast\widetilde{\psi})\,;\quad\varphi,\psi\in\Dc_{d},
\end{equation}
holds, they being linked by
\begin{equation}
\mathrm{tr} K_{U,V}(\varphi)=k_{U,V}(\varphi)\,,\quad \varphi\in \Dc_{d}.
\end{equation}
When $U = V$, then $k_U = k_{U, U}$, respectively $K_U = K_{U, U}$ are (operator) positive definite distributions on $\R^d$.\\
Conversely, for any scalar positive definite distribution $k \in \Dc_d'$ (operator positive definite distribution $K\in \Dc_{d}'\big(\Cs_{1}(H)\big)$) there exists a $\mathbb{R}^{d}$-stationary (respectively operator $\mathbb{R}^d$-stationary) m.s.o.r.d.f. $U\in \Dc_d'(\Hs)$, such that $k = k_U$ (respectively $K=K_{U}$).
\end{teo}

\begin{proof}
Let us treat the operator $\R^d$-stationary case. Let $U,V$ be as in the statement and $C_{U,V}\in \Dc'_{2d}\big(\Cs_{1}(H)\big)$ the operator cross covariance distribution associated to the pair $U,V$ as in Section 3 of \cite{GaPopa}. Fix, for now $h,k\in H$ and $\varphi \in\Dc_{d}$. Then the mapping $T_{\varphi}^{h,k}$ defined by

\begin{equation}\label{4413}
\Dc_{d}\ni\psi\mapsto\left(C_{U,V}(\varphi\otimes\psi)h,k\right)_{H}\,\in\mathbb{C}
\end{equation}
is a distribution on $\mathbb{R}^{d}$, i.e. $T_{\varphi}^{h,k}\in\Dc'_{d}$.
Fixing just $h,k\in H$, the mapping $T^{h,k}$ defined by
$$\Dc_{d}\ni\varphi\mapsto T_{\varphi}^{h,k}\,\in\Dc'_{d}$$
is in $\Dc'_{d}(\Dc'_{d})=\Dc'_{d}\varepsilon\Dc'_{d}$, hence in $\Bc(\Dc_{d},\Dc'_{d})$.
Denote by $\widetilde{T}^{h,k}$ the extension by continuity of $T^{h,k}$ from $\Dc_{d}$ to $\Ec'_{d}$, which is possible due to the density of $\Dc_{d}$ in $\Ec'_{d}$.
Moreover, because of the operator cross $\mathbb{R}^d$~-~stationarity of $U$ and $V$, for any $\psi\in \Dc_{d}$ we successively infer:
\begin{equation*}
\begin{split}
\big(\tau_{x}T_{\varphi}^{h,k}\big)(\psi)&=T_{\varphi}^{h,k}(\tau_{-x}\psi)=
\left(C_{U,V}\big(\varphi\otimes\tau_{-x}\psi\big)h,k\right)_{H}=\\
&=\left(C_{U,V}\big(\tau_{x}\varphi\otimes\tau_{x}\tau_{-x}\psi\big)h,k\right)_{H}=\\
&=\left(C_{U,V}\big(\tau_{x}\varphi\otimes\psi\big)h,k\right)_{H}= T_{\tau_{x}\varphi}^{h,k}(\psi),
\end{split}
\end{equation*}
meaning,
$$\tau_{x}T_{\varphi}^{h,k}=T_{\tau_{x}\varphi}^{h,k}\,,\quad x\in\mathbb{R}^{d},$$
wherefrom
$$\tau_{x}T^{h,k}=T^{h,k}\tau_{x}\,;\quad x\in\mathbb{R}^{d},$$
which holds also for the extension
$\widetilde{T}^{h,k}\in \Bc(\Ec'_{d},\Dc'_{d})$:
$$\tau_{x}\widetilde{T}^{h,k}=\widetilde{T}^{h,k}\tau_{x}\,;\quad x\in\mathbb{R}^{d}.$$

From the reflexivity of $\Ec_{d}$ and $\Dc_{d}$, we infer that $\widetilde{T}^{h,k}$ is the dual of an operator $W^{h,k}\in\Bc(\Ec_{d},\Dc_{d})$, which commutes with the translations.
 Applying Theorem 3.5.19., pp.215 from \cite{DGaPGa} to $W^{h,k}$ we infer that there exists a unique $u^{h,k}\in \Dc'_{d}$ such that

\begin{equation}\label{4414}
W^{h,k}\varphi=u^{h,k}\ast\varphi\,\quad \varphi\in \Dc_{d}.
\end{equation}
Using relations as (3.5.27), (3.5.29) pp. 214 -- 216 of \cite{DGaPGa}, because of (\ref{4414}) and (\ref{4413}) for each $\psi \in \Dc_d$ we successively have \footnote{Here $\check \psi$ means $\check \psi (x) = \psi(-x), x \in \R^d$}:
\begin{multline}\label{eq:corr.distrib}
\big(C_{U,V}(\varphi\otimes\psi)h,k\big)_{H} = T_{\varphi}^{h,k}(\psi) = \widetilde{T}_{\varphi}^{h,k}(\psi) = \left(\big(W^{h,k}\big)'\varphi\right)(\psi) \\
= \big(W^{h,k}\psi\big)(\varphi)
= \Big(\big (u^{h,k}\ast\psi\big)\ast\check{\varphi}\Big)(0) = u^{h,k}(\varphi\ast\check{\psi}).
\end{multline}

Fixing for now $\varphi,\psi\in \Dc_{d}$ in the relation

\begin{equation}\label{4415}
u^{h,k}(\varphi\ast\check{\psi})=\big(C_{U,V}(\varphi\otimes\psi)h,k\big)_{H}\,,\quad h,k\in H,
\end{equation}
the sesquilinearity of the form (\ref{4415}), leads to the existence of a linear operator
$$u:\Dc_{d}\ast\Dc_{d}\mapsto \Cs_{1}(H),$$
which -- because of the continuity of $C_{U,V}$ and of the density of $\Dc_{d}\ast\Dc_{d}$ in $\Dc_{d}$ -- can be extended by continuity to the whole $\Dc_{d}$. Denoting this extension also by $u$, we have that $u\in\Dc'_{d}\big(\Cs_{1}(H)\big)$ satisfies
\begin{equation}\label{4416}
u(\varphi\ast\check{\psi}) = C_{U,V}(\varphi\otimes\psi)\,,\quad \varphi,\psi\in\Dc_{d},
\end{equation}
wherefrom by passing to the cross covariance kernel it results for any $\varphi,\psi\in\Dc_{d}$
\begin{equation}\label{4417}
\Gamma_{U,V}(\varphi,\psi)=C_{U,V}(\varphi\otimes\overline{\psi})=
u(\varphi\ast\check{\overline{\psi}})=u(\varphi\ast\widetilde{\psi}),
\end{equation}
which means that \eqref{eq:corel.oper} holds with $K_{U,V}=u$. \\
It remains to clarify if the $\Cs_{1}(H)$~-~valued distribution $K_{U} = K_{U,U}$ is positive definite. This results by noticing that (\ref{4417}) infers
\begin{equation}\label{4418}
K_{U}(\varphi\ast\widetilde{\varphi})=\Gamma_{U}(\varphi,\varphi)\geqslant0\,,\quad\varphi\in \Dc_{d}
\end{equation}
(see also \cite{Popa2}). The existence of $k_{U,V} \in \Dc'_d$ and \eqref{eq:corel.scal} as well as the positive definiteness of $k_U = k_{U,U}$ are
easily inferred by passing to trace. \\
For the converse let us observe that for a given positive definite $\Cs_1(H)$~-~valued distribution $K$ on $\mathbb{R}^d$, the $\Cs_1(H)$~-~valued distribution $C_K$ on $\mathbb{R}^{2d}$ defined by
$$ C_K (\varphi \otimes\bar\psi) = K (\varphi * \tilde \psi), \quad \varphi, \psi \in \Dc_d,$$
being positive definite is the operator correlation distribution of some m.s.o.r.d.f. $U$, which -- having in view \eqref{eq:link_R^d_D_d} -- is operator $\mathbb{R}^d$~-~stationary.
\end{proof}
\begin{rem}
The formulas \eqref{eq:corel.scal} and \eqref{eq:corel.oper} say in fact that the corresponding cross covariance functions $\gamma_{U,V}$ and $\Gamma_{U,V}$ do not depend separately on $\varphi$ and $\psi$, but on $\varphi * \tilde \psi$. This means that the pair of m.s.o.r.f.d. $U,V$  is (operator) cross ${\mathcal D}_d$~-~stationary.  On the other side (\ref{eq:link_R^d_D_d}) implies now easy that the (operator) cross $\Dc_{d}$~-~stationarity implies the (operator) cross $\mathbb{R}^{d}$~-~stationarity. Moreover, it is now easy to see that the (operator) cross $\R^d~$~-~stationarity is equivalent to the (operator) cross $\Dc_d$~-~stationarity.
\end{rem}
The distributions $k_{U}$ and $K_{U}$ associated to the stationary distribution field $U$ in Theorem \ref{teo441}, will be called the \textit{spectral distribution}, respectively the \textit{operator spectral distribution} of $U$. Since they are positive definite the Bochner-Schwartz type Theorem from \cite{Popa2} can be applied. Thus it holds

\begin{teo}\label{teo442}
Given a (operator) stationary m.s.r.d.f. $U\in \Dc'_{d}(\Hs)$, it's (operator) spectral distribution $k_{U}$ ($K_{U}$ respectively) has the properties:
\begin{itemize}
\item[{(i)}]\begin{eqnarray*}
&\displaystyle\int\big(k_{U}\ast\varphi\ast\widetilde{\varphi}\big)(x)dx\geqq 0\,,\quad\varphi\in \Dc_{d},\\
&\displaystyle\int\big(K_{U}\ast\varphi\ast\widetilde{\varphi}\big)(x)dx\geqslant 0\,,\quad \textrm{in}\ \  \Bc(H)\,\quad\varphi\in \Dc_{d};
\end{eqnarray*}
\item[(ii)]\begin{equation*}
\widetilde{k}_{U}=k_{U}\quad,\quad \widetilde{K}_{U}=K_{U}\,;
\end{equation*}
i.e. the spectral distributions are selfadjoint.
\item[(iii)]\begin{eqnarray*}
&k_{U}\in\Dc'_{d,L^{1}}\subset\Dc'_{d,L^{2}}\subset\Bc'_{d}\subset \Sc'_{d},\\
&K_{U}\in\Dc'_{d,L^{1}}\big(\Cs_{1}(H)\big)\subset\Dc'_{d,L^{2}}\big(\Cs_{1}(H)\big)
\subset\Bc'_{d}\big(\Cs_{1}(H)\big)\subset \Sc'_{d}\big(\Cs_{1}(H)\big);
\end{eqnarray*}
i.e. the spectral distributions are summable and consequently they are also $2$-summable, bounded and tempered.
\item[(iv)] there exists a positive tempered measure $\bnu_{U}$ on $\mathbb{R}^{d}$, (a $\Cs_{1}(H)$~-~valued positive tempered measure $\bF_{U}$ on $\mathbb{R}^{d}$, respectively), such that $k_{U}$ is the Fourier transform of $\bnu_{U}$ and $K_{U}$ is the Fourier transform of $\bF_{U}$.

\end{itemize}

\end{teo}
\begin{proof}
Follows easily, applying the results of \cite{Popa2}.
\end{proof}
In analogy to the classical case from \cite{Kaki}, $\bnu_{U}$ will be called \textit{the representing measure} for $k_{U}$, while $\bF_{U}$ is the \textit{representing measure of the operator spectral distribution} $K_U$.
\begin{defn}
The spaces $L^{2}(\bnu_{U})$ and $L^{2}(\bF_{U})$, associated to the representing measures of the covariance distributions $k_{U}$ and $K_{U}$ of the stationary m.s.o.r.d.f. $U$, are called \textit{the vector spectral domain} (\textit{modular spectral domain} respectively) of the m.s.o. distribution field $U$.
\end{defn}
For the Fourier transform $\mathcal{F}^{1}_{d}$ and for its inverse $\mathcal{\widetilde{F}}_{d}^{1}$ on space $L^{1}(\mathbb{R}^{d})$ of complex valued Lebesque integrable functions on $\mathbb{R}^{d}$, we refer to \cite{Schw1} Def.(VII.2;2) and (VII.2;3) pp.87. Denoting by $\Fc_{d}$, $\widetilde{\Fc}_{d}$ their respective restrictions to the space $\Sc_{d}:=\Sc(\mathbb{R}^{d})$ of complex valued indefinite derivable rapidly decreasing functions from distribution theory $(\Sc_{d}\supset\Dc_{d})$  notice that both are linear topological space automorphisms of $\Sc_{d}$, while their adjoints $\Fc'_{d}$ and $\widetilde{\Fc}'_{d}$ are automorphisms of the space $\Sc'_{d}$ of complex valued tempered distributions, $\Sc'_{d}\subset\Dc'_{d}$.\\
The corresponding Fourier transforms on the spaces $\Sc'_{d}(\Hs):=\Bc(\Sc_{d},\Hs)=\Sc'_{d}\widehat{\otimes}_{\varepsilon}\Hs\subset \Dc'_{d}(\Hs)$ are defined as follows
\begin{equation*}
\bFc'_{d}:=\Fc'_{d}\otimes I_{\Hs},\quad \overline{\bFc}'_{d}:=\overline{\Fc}'_{d}\otimes I_{\Hs}
\end{equation*}
\begin{rem}
Under these hypotheses we can state that
the Fourier transform of a tempered m.s.o.r.d.f. is a tempered m.s.o.r.d.f. too.
\end{rem}
Now we can prove the extension on our framework of the Kolmogorov isomorphism Theorem.
\begin{teo}\label{teo443}
The (operator) stationary m.s.o.r.d.f. $U=\{U_{\varphi}\}_{\varphi\in\Dc_{d}}\in\Dc'_{d}(\Hs)$, is (gramian) unitary equivalent to the Fourier transformation on $L^{2}(\bnu_{U})$ ( $L^{2}(\bF_{U})$ respectively). More precisely, if $U$ is stationary, respectively operator stationary, then the mapping
\begin{equation}\label{4419}
\Hs_{(U)}\ni U_{\varphi} \mapsto \Fc_d\varphi \in L^{2}(\bnu_{U})\,,\quad \varphi\in \Dc_{d}
\end{equation}
can be extended to a Hilbert space isomorphism between $\Hs_{(U)}$ and $L^{2}(\bnu_{U})$, respectively the mapping
\begin{equation}\label{4420}
\Hs_{U}\ni U_{\varphi}\mapsto \Fc_d\varphi I_{H}\in L^{2}(\bF_{U})\,,\quad \varphi\in \Dc_{d}
\end{equation}
can be extended to a normal Hilbert $\Bc(H)$~-~module isomorphism of $\Hs_U$ on $L^2(\bF_U)$.
\end{teo}
\begin{proof}
Let $U$ be operator stationary. The successive computations, using Theorem \ref{teo442} (v) and Theorem \ref{teo441}
\begin{equation*}
\begin{split}
\big[U\varphi,U\psi\big]_{\Hs}&=K_{U}(\varphi\ast\widetilde{\psi})=\bFc'_d \bF_{U}(\varphi\ast\widetilde{\psi})=\\
&=\int\limits_{\mathbb{R}^{d}}\Fc_d(\varphi\ast\widetilde{\psi})(t)d \bF_{U}(t)=\int\limits_{\mathbb{R}^{d}}\Fc_d(\varphi)\overline{\Fc_d\psi}d \bF_{U}(t)=\\
&=\big[\Fc_d\varphi I_{H},\Fc_d(\psi)I_{H}\big]_{L^{2}(\bF_{U})}
\end{split}
\end{equation*}
lead to the conclusion that the modular, gramian preserving, extension of the mapping (\ref{4420}) yields the gramian unitary operator ${\mathbf K}_{U}:\Hs_{U}\rightarrow L^{2}(\bF_{U})$ for which
\begin{equation}\label{4421}
{\mathbf K}_{U}U_{\varphi}=\Fc_{d}(\varphi)I_{H}\,;\quad \varphi\in\Dc_{d}.
\end{equation}
Analogously (\ref{4419}) can be linearly continuously extended to an isomorphism
${\mathbf K}_{(U)}:\Hs_{(U)}\rightarrow L^{2}(\bnu_{U})$ for which
\begin{equation}\label{4422}
{\mathbf K}_{(U)}U_{\varphi}=\Fc_{d}(\varphi)\,;\quad \varphi\in\Dc_{d}.
\end{equation}
\end{proof}
We are now in a position to state the analogue of the integral representation theorem of a (operator) stationary m.s.o.r.d.f. More precisely, we give the representation of such a field as the Fourier transform of a (gramian) orthogonally scattered multivariate stochastic (not necessarily bounded) measure.

\begin{teo}\label{teo444}
Let $U=\{U_{\varphi}\}_{\varphi\in\Dc_{d}}\in\Dc'_{d}(\Hs)$.
\begin{enumerate}
\item[(i)] $U$ is a stationary m.s.o.r.d.f., iff there exists a $\delta$~-~ring $\mathfrak{D}_{(U)}$ of Borel subsets of $\mathbb{R}^{d}$ and a tempered m.s.o. stochastic measure $\xi_{(U)}:\mathfrak{D}_{(U)}\to \Hs$, which is ${\mathfrak D}_{(U)}$ stationary, such that
\begin{equation}\label{4423}
U=\bFc'_d\xi_{(U)}.
\end{equation}
If (\ref{4423}) holds, then
\begin{enumerate}
\item[(i.1)]the vector time domains of $\xi_{(U)}$ and of $U$ are the same,i.e. $\Hs_{\xi_{(U)}}=\Hs_{(U)}$,
\item[(i.3)]$\bnu_{U}(\cdot)=\|\xi_{(U)}(\cdot)\|^{2}_{\Hs}$.
\end{enumerate}
\item[(ii)]$U$ is an operator stationary m.s.o.r.d.f., iff there exists a $\delta$~-~ring $\mathfrak{D}_{U}$ of Borel subsets of $\mathbb{R}^{d}$ and a tempered m.s.o. stochastic measure  $\xi_{(U)}:\mathfrak{D}_{U}\to \Hs$, which is operator ${\mathfrak D}_{U}$ stationary such that

\begin{equation}\label{4424}
U=\bFc'_d\xi_{U}
\end{equation}
holds.\\
If (\ref{4424}) holds, then
\begin{enumerate}
\item[(ii.1)] the modular time domain of $\xi_U$ and of $U$ are the same, i.e. $\Hs_{\xi_{U}}=\Hs_{U}$;
\item[(ii.3)]$\bF_{U}(\cdot)=[\xi_{U}(\cdot),\xi_{U}(\cdot)]_{\Hs}$.
\end{enumerate}
\end{enumerate}
\end{teo}
\begin{proof}
We restrain ourselves to the case where $U$ is operator stationary.
Let ${\mathbf K}_{U}$ be the Kolmogorov isomorphism of $\Hs_{U}$ onto $L^{2}(\bF_{U})$, where $\bF_{U}$ is the representing measure of the covariance kernel $\Gamma_{U}$ of $U$. Denote by $\mathfrak{D}_{U}$ the $\delta$~-~ring of Borel sets $\beta$ from $\mathbb{R}^{d}$ for which $\|\bF_{U}\|_{o}(\beta)<+\infty$. For such $\beta$, it is obvious that $\chi_{\beta}I_{H}\in L^{2}(\bF_{U})$. We shall consider hence

\begin{equation}\label{4425}
\xi_{U}(\beta):={\mathbf K}_{U}^{-1}(\chi_{\beta}I_{H})\,,\quad \beta\in\mathfrak{D}_{U}.
\end{equation}
The properties of $\bF_{U}$ and of ${\mathbf K}_{U}$, imply $\xi_{U}(\emptyset)=0$ and that $\xi_{U}$ is completely additive. Since ${\mathbf K}_{U}$ is a gramian unitary operator between $\Hs_{U}$ and $L^{2}(\bF_{U})$ we shall have for $\beta_{1},\,\beta_{2}\in \mathfrak{D}_{U}$

\begin{equation}\label{4426}
\begin{split}
\big[\xi_{U}(\beta_{1}),\xi_{U}(\beta_{2})\big]_{\Hs}&=\int\limits_{\mathbb{R}^{d}}
\chi_{\beta_{1}}(t)\overline{\chi_{\beta_{2}}(t)}I_{H}d\bF_{U}(t)=\\
&=\bF_{U}(\beta_{1}\cap\beta_{2}),
\end{split}
\end{equation}
wherefrom we infer $\xi_{U}\in cagos(\mathfrak{D},\Hs_{U})$. From the definition of $\xi_{U}$, putting $f=\sum\limits_{j=1}^{n}\chi_{\beta_{j}}a_{j}$, $ a_{j}\in \Bc(H),\,\beta_{j}$~-~disjoint, since ${\mathbf K}_{U}$ (and ${\mathbf K}_{U}^{-1}$) is $\Bc(H)$~-~linear we get:

\begin{equation}\label{4427}
K_{U}^{-1}(f)=\int\limits_{\mathbb{R}^{d}}f(t)d\xi_{U}(t),
\end{equation}
which can be extended by taking the limits for $f\in L^{2}(\bF_{U})$. Putting now $f=\Fc\varphi I_{\Hs}$, (\ref{4427}) becomes $U_{\varphi}=\int\limits_{\mathbb{R}^{d}}\Fc\varphi d\xi_{U}(t)\,,\quad (\varphi\in \Dc_{d})$, which is exactly (\ref{4424}).\\
Now (ii) results from the integral representation of $U_{\varphi}$, given by (\ref{4424}) $U_{\varphi}=\int\limits_{\mathbb{R}^{d}}(\Fc\varphi)(t)d\xi_{U}(t)\,,\quad \varphi\in \Dc_{d}$, and (ii.2) is equivalent to the fact that $\xi_{U}$ is gramian orthogonally scattered (gos), while (ii.3) results, computing the gramian $[U_{\varphi},U_{\psi}]=\Gamma_{U}(\varphi,\psi)$, using the integral forms of $U_{\varphi}$ and $U_{\psi}$ and then identifying with the integral representation of the operator covariance by  $\bF_{U}$.
\end{proof}


\begin{thebibliography}{99}




\bibitem{Bala}{{\sc K. Balagangadharan}, {\em The prediction theory of stationary random distributions}, {Memoirs of the College of Science}, {University of Kyoto}, {Series A, Mathematics} vol. XXXIII, {no. 2}, (1960), 243 -- 256.}














\bibitem{DGaPGa}{{\sc D. Ga\c{s}par, P. Ga\c{s}par}, {\em Analiz\u{a} func\c{t}ional\u{a}}, Ed. a 2-a,  Ed. de Vest, Timi\c{s}oara, 2009.}



























\bibitem{GaPopa} {{\sc P. Ga\c spar, L. Popa}, {\em Stochastic Mappings and Random Distribution Fields. A Correlation Approach}, {to appear.}}




\bibitem{Gelf1}{{\sc I. M. Gelfand}, {\em Generalized random processes  (russian)}, DAN 100 (1955), 953 -- 856.}













\bibitem{Ito} { {\sc K. It\^o}, {\em Stationary Random Distributions},{ Mem. Coll. Sci. Univ. Kyoto A}, XXVIII {\bf 3}, (1953), 209 -- 223.}

\bibitem{Kaki}{ {\sc Y. Kakihara}, {\em Multidimensional Second Order Stochastic Processes}, {World Scientific Publ.Comp.}, {River Edge}, N.I., 1997.}




\bibitem{Khin} {{\sc A.Ya. Khintchine}, {\em Korrelationstheorie der station\"{a}ren stochastischen Prozesse}, {Math. Ann.} {\bf109} (1934), 605 -- 615.}

\bibitem{Kolmo}{ {\sc A.~N. Kolmogorov}, {\em Stationary sequences in Hilbert space}, {Bull. Moskov. Gos. Univ. Matematika}, {\bf 2}, (1941), 1 -- 40.}




\bibitem{MasDilat} {{\sc P. Masani}, {\em Dilations as propagators of Hilbertian varieties},{ Siam  J. Math. Anal., Vol.9},  {\bf 3}  {June}  (1978), 414 -- 456.}
\bibitem{MasPred} {{\sc P. Masani}, {\em The prediction theory of multivariate stochastic processes}, {III},{ Acta Math.},  {\bf 104}  {June}  (1960), 141 -- 162.}











\bibitem{Popa2} {{\sc L. Popa}, {\em A characterization of operator positive definite distributions},{ Anal. Univ. Vest Timi\c soara}, vol. XLV fasc. {\bf 2}, (2007), 109 -- 116.}









\bibitem{Roz2}{ {\sc Yu. A. Rozanov}, {\em On extrapolation of random distributions (russian)}, {Teorija Veroyatnosti IV} (1959), 465 -- 471.}




\bibitem{Schw1} {{\sc L. Schwartz}, {\em Th\'eorie des distributions I, II}, Fasc. IX, X, {Actualit\'{e}s Scientifiques Industrielles} {\bf 1245, 1112}, Hermann, Paris, 1957, 1959.}








\bibitem{SucVal} {{\sc I. Suciu and I. Valu\c{s}escu}, {\em A linear filtering problem in complet correlated actions},{ Journal of Multivariate Analysis} {\bf 9, 4} (1979), 559 -- 613.}




\bibitem{Szaf}{{\sc F. H. Szafraniec}, {\em Murphy's positive definite kernels and Hilbert $C^{\ast}$~-~modules reorganized, Noncommutative Harmonic Analysis with Applications to Probability} II, Banach Center Publications, Inst. of Math. Polysh Acad. of Sciences, Warszawa, vol. 89 (2010), 275 -- 295}






\bibitem{Val 80} {{\sc I. Valu\c{s}escu}, {\em Stationary processes in complete correlated actions}, {Mon. Mat.} {\bf 80}, {West University of  Timi\c{s}oara}, {Timi\c{s}oara}, 2007.}


\bibitem{ValPas}  {{\sc I. Valu\c{s}escu and P. Ga\c{s}par}, {\em On Uniformly Bounded Linearly $\Gamma$~-~stationary Processes}, {Numerical Analysis and Applied Mathematics},{International Conference}, {Rhodes}, {Greece}, 2010}



\bibitem{WieMas1} {{\sc N. Wiener, P. Masani}, {\em The prediction theory of multivariate stochastic processes, I},{ Acta. Math.},{\bf 98} (1957), 111 -- 150.}

\bibitem{WieMas2} {{\sc N. Wiener, P. Masani}, {\em The prediction theory of multivariate stochastic processes, II}, {Acta. Math.}, {\bf 99} (1958), 93 -- 137.}

\bibitem{Wold} {{\sc H. Wold}, {\em A Study in the Analysis of Stationary Time Series}, {Almquist Wiksell}, {Stokholm}, 1938, $2^{-nd}$ ed., Stockholm, 1954.}








\end{thebibliography}
\end{document}